\documentclass[A4paper,11pt]{amsart}

\usepackage[english]{babel}
\usepackage{amsmath}
\usepackage{amssymb}
\usepackage{amsfonts}
\usepackage[all]{xy}
\usepackage{hyperref}
\usepackage{color}
\usepackage{lscape}

 \newtheorem{proposition}{Proposition}[section]
 \newtheorem{lemma}[proposition]{Lemma}
 \newtheorem*{lemma*}{Lemma}
 \newtheorem{theorem}[proposition]{Theorem}
 
 \theoremstyle{definition}
 
 \newtheorem*{definition*}{Definition}
 
 \newtheorem*{example*}{Example}

\numberwithin{equation}{section}

\def\<{\langle}
\def\>{\rangle}

\definecolor{lightgray}{rgb}{0.666666,0.666666,0.666666}

\begin{document}

\title{A class of twisted partial Hopf actions}
\author{Quanguo Chen }
\address{College of Mathematics and Statistics, Kashi University, Xinjiang, Kashi, China\\}
\email{cqg211@163.com (Q.G.Chen)}
\thanks{This work was supported by the National Natural Science
Foundation of China (No. 12271292) and the Natural Foundation of Shandong Province(No. ZR2022MA002).}
\begin{abstract}
In this paper, we will introduce a novel method for constructing numerous examples of twisted partial Hopf actions. Utilizing split quaternions, split semi-quaternions, and $\frac{1}{4}$-quaternions as our subjects of study, we have obtained a class of twisted partial Hopf actions by examining the twisted partial actions of Sweedler Hopf algebra on these structures.
\end{abstract}
\maketitle
\section{ Introduction }
\def\theequation{1. \arabic{equation}}
\setcounter{equation} {0} \hskip\parindent

 The concept of a partial Hopf action emerged as a generalization of traditional Hopf actions, first introduced and analyzed in  \cite{Caenepeel}. Partial Hopf actions have important applications in non-commutative geometry, quantum groups, $C^{\ast}$-algebras, and dynamical systems. Extending this framework, the notion of twisted partial Hopf actions was formalized in  \cite{Alves4}.  Further information around partial actions may be consulted in the survey \cite{Alves1}, \cite{Alves2}, \cite{Alves3} and \cite{Dokuchaev3}.
 
 Quaternions were first proposed by William Rowan Hamilton in 1843 as a non-commut
 ative extension of complex numbers into four-dimensional space\cite{Hamilton}, and are widely applied in fields such as geometry, physics, and computer graphics. Split quaternions, proposed by Sir James Cockle, and split semi-quaternions, alternatively termed degenerate split quaternions, further generalize these structures in \cite{Cockle} and \cite{Rosenfeld}.

In \cite{Alves4}, the authors leverage the interplay between algebraic groups and commutative Hopf algebras to elucidate methods for generating twisted partial Hopf actions. Such constructions are inherently intricate. This article aims to establish a class of twisted partial Hopf actions, providing a foundational framework for further exploration.

The paper is organized as follows.

In Section 2, we recall some basic definitions and results concerning split quaternions, split semi-quaternios, $\frac{1}{4}$-quaternions, Sweedler Hopf algebra and twisted partial Hopf action. In Section 3, we present explicit constructions of twisted partial actions of the Sweedler Hopf algebra on split semi-quaternions, while Sections 4 and 5 extend these results to split quaternions and $\frac{1}{4}$-quaternions, respectively. Section 6 concludes with a discussion of potential applications and future research directions.

\section{Preliminaries}
\def\theequation{2. \arabic{equation}}
\setcounter{equation} {0} \hskip\parindent

Throughout the paper, $\mathbb{R}$ denotes the real number field. All algebras and coalgebras are defined over $\mathbb{R}$ and  linear means $\mathbb{R}$-linear.  Tensor products without explicit notation are assumed to be over  $\mathbb{R}$. For a coalgebra $(C, \Delta, \varepsilon)$, we use Sweedler-Heynemann notation with summation sign suppressed for the comultiplication map $\Delta$, namely
$$
\Delta(c)=c_{1}\otimes c_{2},
$$
for all $c \in C$.  For further details on Hopf algebras, see \cite{Abe} and \cite{Montgomery}.

\subsection{Split quaternions}

An element  $q$  of the split quaternion algebra $\mathbb{H}_{s}$ takes
the form
$q=a_{0}1+a_{1}e_{1}+a_{2}e_{2}+a_{3}e_{3},$
where $a_{0}, a_{1}, a_{2}, a_{3}\in \mathbb{R}$ and $e_{1}, e_{2}, e_{3}$ satisfy the following relations:
$$
e_{1}^{2}=-1, e_{2}^{2}=1, e_{3}^{2}=1, e_{1}e_{2}=e_{3}=-e_{2}e_{1}, e_{2}e_{3}=- e_{1}=-e_{3}e_{2}, e_{3}e_{1}= e_{2}=-e_{1}e_{3}.
$$
Under these rules, split quaternions form an associative and distributive algebra.

\subsection{Split semi-quaternions}

An element  $q$  of the split semi-quaternion algebra $\mathbb{H}_{ss}$ is expressed as 
$q=a_{0}1+a_{1}e_{1}+a_{2}e_{2}+a_{3}e_{3},$
where $a_{0}, a_{1}, a_{2}, a_{3}\in \mathbb{R}$ and $e_{1}, e_{2}, e_{3}$ obeying:
$$
e_{1}^{2}=1, e_{2}^{2}=0, e_{3}^{2}=0, e_{1}e_{2}=e_{3}=-e_{2}e_{1}, e_{2}e_{3}=0=-e_{3}e_{2}, e_{3}e_{1}= -e_{2}=-e_{1}e_{3}.
$$
These rules ensure that split semi-quaternions constitute an associative and distributive algebra.

\subsection{$\frac{1}{4}$-quaternions}

A $\frac{1}{4}$-quaternion $q$ is an expression of
the form
$q=a_{0}+a_{1}e_{1}+a_{2}e_{2}+a_{3}e_{3},$
where $a_{0}, a_{1}, a_{2}, a_{3}\in \mathbb{R}$  and the basis elements satisfy:
$$
e_{1}^{2}=0, e_{2}^{2}=0, e_{3}^{2}=0, e_{1}e_{2}=e_{3}=-e_{2}e_{1}, e_{2}e_{3}=0=e_{3}e_{2}, e_{3}e_{1}=0=e_{1}e_{3}.
$$
The set of all $\frac{1}{4}$-quaternions is denoted  $\mathbb{H}_{00}$. With the specified addition and multiplication, $\mathbb{H}_{00}$ forms an associative algebra, referred to as the $\frac{1}{4}$-quaternion algebra.
\subsection{Sweedler Hopf Algebra} Sweedler Hopf algebra $\mathbb{H}_{4}$ is generated by two elements
$g$ and $\nu$  satisfying:
$$
g^{2}=1, \nu^{2}=0, g\nu+\nu g=0.
$$
Its comultiplication,  antipode and  counit of $\mathbb{H}_{4}$ are given by
$$
\Delta(g)=g\otimes g, \Delta(\nu)=g\otimes\nu+\nu\otimes 1, \varepsilon(g)=1, \varepsilon(\nu)=0, S(g)=g, S(\nu)=-g\nu.
$$
The dimension of $\mathbb{H}_{4}$ is four, with $1,g,\nu, g\nu$ forming a basis.

\subsection{Twisted partial actions}

Let $H$ be a Hopf algebra, $A$ a unital algebra with unit $1_{A}$. Let $\cdot: H\otimes A\rightarrow A$ and
$\omega: H\otimes H\rightarrow A$ be two linear maps.  Denote  $\omega(h\otimes l)=\omega(h,l)$, where $h,l\in H$.

The pair $(\cdot, \omega)$ is termed a twisted partial action of $H$ on $A$, if the following conditions hold:
$\forall a, b\in A$ and $h,l\in H$,
\begin{equation}\label{E1}
	1_{H}\cdot a=a,
\end{equation}
\begin{equation}\label{E2}
	h\cdot(ab)=(h_{1}\cdot a)(h_{2}\cdot b),
\end{equation}
\begin{equation}\label{E3}
(h_{1}\cdot (l_{1}\cdot a))\omega(h_{2}, l_{2})=\omega(h_{1}, l_{1})(h_{2}l_{2}\cdot a),
\end{equation}
\begin{equation}\label{E4}
\omega(h, l)=\omega(h_{1}, l_{1})(h_{2}l_{2}\cdot 1_{A}).
\end{equation}

Given a twisted partial action $(H, A, \cdot, \omega)$,  a product is defined on the vector space $A\otimes H$:
$$
(a\otimes h)(b\otimes l)=a(h_{1}\cdot b)\omega(h_{2},l_{1})\otimes h_{3}l_{2},
$$
for all $a, b\in A, h, l\in H$. Define
 $A\sharp _{(\cdot,\omega)}H=(A\otimes H)(1_{A}\otimes 1_{H})$, with $a\sharp h=a(h_{1}\cdot 1_{A})\otimes h_{2}$. 
In general, $A\otimes H$ with the above defined product is neither associative nor unital.  Necessary and sufficient conditions for $A\sharp _{(\cdot,\omega)}H$ to be associative and unital with unit $1_{A}\sharp 1_{H}$ are: 
\begin{equation}\label{E5}
	\omega(h,1_{H})=\omega(1_{H},h)=h\cdot 1_{A},
\end{equation} 
\begin{equation}\label{E6}
	(h_{1}\cdot \omega(m_{1},t_{1}))\omega(h_{2}, m_{2}t_{2})=\omega(h_{1},m_{1}))\omega(h_{2}m_{2}, t),
\end{equation} 
for all $h, m, t\in H$. The algebra $A\sharp _{(\cdot,\omega)}H$ is called a partial crossed product. 

\section{Twisted partial actions of $\mathbb{H}_{4}$ on $\mathbb{H}_{ss}$}

We now construct explicit examples of twisted partial actions of the Sweedler Hopf algebra $\mathbb{H}_{4}$ on the split semi-quaternion algebra $\mathbb{H}_{ss}$.

Define two linear maps $\cdot: \mathbb{H}_{4}\otimes \mathbb{H}_{ss}\rightarrow  \mathbb{H}_{ss} $ and $\omega: \mathbb{H}_{4}\otimes\mathbb{H}_{4}\rightarrow \mathbb{H}_{ss}$ via the following table: 
$$\begin{array}{|c|c|c|c|c|}
	\hline   \cdot                           & 1 & e_{1} &e_{2}&e_{3} \\
	\hline 1                 & 1 & e_{1} &e_{2}&e_{3} \\
	\hline  g                     & 0& 0 & 0 & 0 \\
	\hline  \nu                   &k_{1}1+k_{2}e_{1}+k_{3}e_{2}-k_{4}e_{3}& k_{2}1+k_{1}e_{1}+k_{4}e_{2}-k_{3}e_{3}&  k_{1}e_{2}+k_{2}e_{3} & k_{2}e_{2}+k_{1}e_{3} \\
	\hline  g\nu                    & l_{1}1+l_{2}e_{1}+l_{3}e_{2}+l_{4}e_{3}&l_{2}1+l_{1}e_{1}+l_{4}e_{2}+l_{3}e_{3}& l_{1}e_{2}-l_{2}e_{3}& -l_{2}e_{2}+l_{1}e_{3}   \\
	\hline
\end{array}$$
and for $\omega$, the values are defined as:
$$
\begin{cases}
	\omega(1,1)=1, \omega(1,g)=0, \omega(1,\nu)=k_{1}1+k_{2}e_{1}+k_{3}e_{2}-k_{4}e_{3},\\ \omega(1,g\nu)=l_{1}1+l_{2}e_{1}+l_{3}e_{2}+l_{4}e_{3},&\\
	\omega(g,1)=	\omega(g,g)=	\omega(g,\nu)=	\omega(g,g\nu)=0, &\\
	\omega(\nu,1)=k_{1}1+k_{2}e_{1}+k_{3}e_{2}-k_{4}e_{3}, \omega(\nu,g)=0,\\
	\omega(\nu,\nu)=(k^{2}_{1}+k^{2}_{2})1+2k_{1}k_{2}e_{1}+2k_{1}k_{3}e_{2}-2k_{1}k_{4}e_{3},\\ \omega(\nu,g\nu)=(k_{1}l_{1}+k_{2}l_{2})1+(k_{2}l_{1}+k_{1}l_{2})e_{1}+(k_{3}l_{1}+k_{4}l_{2}+k_{1}l_{3}+k_{2}l_{4})e_{2}\\
	\hspace{2cm}+(-k_{4}l_{1}-k_{3}l_{2}+k_{2}l_{3}+k_{1}l_{4})e_{3}, &\\
	\omega(g\nu,1)=l_{1}1+l_{2}e_{1}+l_{3}e_{2}+l_{4}e_{3}, \omega(g\nu,g)=\omega(g\nu,g\nu)=0,\\ \omega(g\nu,\nu)=(k_{1}l_{1}+k_{2}l_{2})1+(k_{2}l_{1}+k_{1}l_{2})e_{1}+(k_{3}l_{1}+k_{4}l_{2}+k_{1}l_{3}+k_{2}l_{4})e_{2}\\
	\hspace{2cm}+(-k_{4}l_{1}-k_{3}l_{2}+k_{2}l_{3}+k_{1}l_{4})e_{3}, &
\end{cases}
$$
where $k_{i}, l_{j}(i=1,2,3,4, j=1,2,3,4)$ are parameters. 

\begin{theorem}
	With the above-defined maps $\cdot$ and $\omega$, the pair $(\cdot, \omega)$ is a twisted partial action of $\mathbb{H}_{4}$ on $\mathbb{H}_{ss}$. 
\end{theorem}
\begin{proof}
It suffices to verify that	$(\cdot, \omega)$ satisfies the twisted partial action conditions (\ref{E1})-(\ref{E4}). While the full proof is tedious, we demonstrate a partial verification. 
	$$
	\nu \cdot (e_{1}e_{2})=\nu \cdot e_{3}=k_{2}e_{2}+k_{1}e_{3}
	$$
and
	\begin{eqnarray*}
		(\nu _{1}\cdot e_{1})	(\nu _{2}\cdot e_{2})
		&=&	(g\cdot e_{1})	(\nu \cdot e_{2})+(\nu\cdot e_{1})	(1 \cdot e_{2})\\
		&=&	(k_{2}1+k_{1}e_{1}+k_{4}e_{2}-k_{3}e_{3})  e_{2}\\
	&=&	k_{2}e_{2}+k_{1}e_{1}e_{2}+k_{4}e_{2}e_{2}-k_{3}e_{3}e_{2}\\
		&=&	k_{2}e_{2}+k_{1}e_{3},
	\end{eqnarray*}
	which confirms that the condtion (\ref{E2}) holds for $\nu, e_{1}, e_{2}$. A similar verification applies to $g\nu, e_{1}, e_{2}$: 
	$$
	g\nu \cdot (e_{1}e_{2})=g\nu \cdot e_{3}=-l_{2}e_{2}+l_{1}e_{3}
	$$
	and
	\begin{eqnarray*}
		((g\nu) _{1}\cdot e_{1})	((g\nu) _{2}\cdot e_{2})
		&=&	(1\cdot e_{1})	(g\nu \cdot e_{2})+(g\nu\cdot e_{1})	(g \cdot e_{2})\\
		&=&	e_{1}(l_{1}e_{2}-l_{2}e_{3})\\
		&=&	l_{1}e_{1}e_{2}-l_{2}e_{1}e_{3}\\
		&=&	l_{1}e_{3}-l_{2}e_{2}. 
	\end{eqnarray*}
For the condition (\ref{E4}), consider:
	\begin{eqnarray*}
		\omega(\nu_{1}, (g\nu)_{1})(\nu_{2}(g\nu)_{2}\cdot 1)
		&=& 	\omega(g, g\nu)(\nu g\cdot 1)+\omega(\nu, 1)(g\nu\cdot 1)+\omega(\nu, g\nu)(g\cdot 1)\\
		&=& \omega(\nu, 1)(g\nu\cdot 1)\\
		&=& (k_{1}1+k_{2}e_{1}+k_{3}e_{2}-k_{4}e_{3})(l_{1}1+l_{2}e_{1}+l_{3}e_{2}+l_{4}e_{3})\\
		&=& (k_{1}l_{1}+k_{2}l_{2})1+(k_{1}l_{2}+k_{2}l_{1})e_{1}+(k_{1}l_{3}+k_{2}l_{4}+k_{3}l_{1}+k_{4}l_{2})e_{2}\\
	&&	+(k_{1}l_{4}+k_{2}l_{3}-k_{3}l_{2}-k_{4}l_{1})e_{3}\\
	&=&\omega(\nu, g\nu).
	\end{eqnarray*}
	This demonstrates the validity of condition (\ref{E4}).
	Since \begin{eqnarray*}
		&&	(\nu_{1}\cdot ((g\nu)_{1}\cdot e_{1}))\omega(\nu_{2}, (g\nu)_{2})\\
		&=&(g\cdot (1\cdot e_{1}))\omega(\nu, g\nu)+(g\cdot (g\nu\cdot e_{1}))\omega(\nu, g)\\
		&&+(\nu\cdot (1\cdot e_{1}))\omega(1, g\nu)+(\nu\cdot (g\nu\cdot e_{1}))\omega(1, g)\\
		&=& (\nu\cdot  e_{1})\omega(1, g\nu)\\
		&=&(k_{2}1+k_{1}e_{1}+k_{4}e_{2}-k_{3}e_{3})(l_{1}1+l_{2}e_{1}+l_{3}e_{2}+l_{4}e_{3})\\
		&=&(k_{2}l_{1}+k_{1}l_{2})1+(k_{2}l_{2}+k_{1}l_{1})e_{1}+(k_{2}l_{3}+k_{1}l_{4}+k_{4}l_{1}+k_{3}l_{2})e_{2}\\
		&&+(k_{2}l_{4}+k_{1}l_{3}-k_{4}l_{2}-k_{3}l_{1})e_{3}
	\end{eqnarray*}
	and 
	\begin{eqnarray*}
		\omega(\nu_{1},(g\nu)_{1})(\nu_{2}(g\nu)_{2}\cdot e_{1})
		&=&\omega(g,g\nu)(\nu g\cdot e_{1})+\omega(\nu,1)( g\nu\cdot e_{1})+\omega(\nu,g \nu)( g\cdot e_{1})\\
		&=&\omega(\nu,1)( g\nu\cdot e_{1})\\
		&=&(k_{1}1+k_{2}e_{1}+k_{3}e_{2}-k_{4}e_{3})(l_{2}1+l_{1}e_{1}+l_{4}e_{2}+l_{3}e_{3})\\
		&=&(k_{1}l_{2}+k_{2}l_{1})1+(k_{1}l_{1}+k_{2}l_{2})e_{1}+(k_{1}l_{4}+k_{2}l_{3}+k_{3}l_{2}+k_{4}l_{1})e_{2}\\
		&&+(k_{1}l_{3}+k_{2}l_{4}-k_{3}l_{1}-k_{4}l_{2})e_{3},
	\end{eqnarray*}
	it follows that 
	$$
	\omega(\nu_{1}, (g\nu)_{1})(\nu_{2}(g\nu)_{2}\cdot 1)=\omega(\nu_{1},(g\nu)_{1})(\nu_{2}(g\nu)_{2}\cdot e_{1}),
	$$
	which demonstrates the validity of condition (\ref{E4}) for $\nu, g\nu, e_{1}$. The detailed computations for other cases follow analogous reasoning.
\end{proof}

\begin{lemma}\label{F1}
With $\cdot$ and $\omega$ defined as above, conditions (\ref{E5}) and (\ref{E6}) hold. 
\end{lemma}
\begin{proof}
	Condition (\ref{E5})  is trivially satisfied. Since 
	\begin{eqnarray*}
		& &	(\nu_{1}\cdot \omega(\nu_{1},(g\nu)_{1}))\omega(\nu_{2}, \nu_{2}(g\nu)_{2})\\
		&=& 	(g\cdot \omega(g,g\nu)\omega(\nu, \nu g)+	(g\cdot \omega(\nu,1)\omega(\nu, g\nu )+	(g\cdot \omega(\nu,g\nu)\omega(\nu, g )\\
		&&+	(\nu\cdot \omega(g,g\nu)\omega(1, \nu g)+	(\nu\cdot \omega(\nu,1)\omega(1, g\nu )+	(\nu\cdot \omega(\nu,g\nu)\omega(1, g )\\
		&=& 	(\nu\cdot \omega(\nu,1)\omega(1, g\nu )\\
		&=&	(k_{1}(\nu\cdot 1)+k_{2}(\nu\cdot e_{1})+k_{3}(\nu\cdot e_{2})-k_{4}(\nu\cdot e_{3})))(l_{1}1+l_{2}e_{1}+l_{3}e_{2}+l_{4}e_{3})\\
		&=&(\left(k_1^2+k_2^2\right) l_1+2 k_1 k_2 l_2)1+(2 k_1 k_2 l_1+\left(k_1^2+k_2^2\right) l_2)e_{1}\\
		&&+(2 k_1 k_3 l_1+2 k_1 k_4 l_2+\left(k_1^2+k_2^2\right) l_3+2 k_1 k_2 l_4)e_{2}\\
		&&+(-2 k_1 k_4 l_1-2 k_1 k_3 l_2+2 k_1 k_2 l_3+\left(k_1^2+k_2^2\right) l_4)e_{3}
	\end{eqnarray*}
	and
	\begin{eqnarray*}
		&&\omega(\nu_{1},\nu_{1})\omega(\nu_{2}\nu_{2}, g\nu)\\
		&=&\omega(g,\nu)\omega(\nu, g\nu)+\omega(\nu,g)\omega(\nu, g\nu)+\omega(\nu,\nu)\omega(1, g\nu)\\
		&=&\omega(\nu,\nu)\omega(1, g\nu)\\
		&=&((k^{2}_{1}+k^{2}_{2})1+2k_{1}k_{2}e_{1}+2k_{1}k_{3}e_{2}-2k_{1}k_{4}e_{3})(l_{1}1+l_{2}e_{1}+l_{3}e_{2}+l_{4}e_{3})\\
		&=&(\left(k_1^2+k_2^2\right) l_1+2 k_1 k_2 l_2)1+(2 k_1 k_2 l_1+\left(k_1^2+k_2^2\right) l_2)e_{1}\\
		&&+(2 k_1 k_3 l_1+2 k_1 k_4 l_2+\left(k_1^2+k_2^2\right) l_3+2 k_1 k_2 l_4)e_{2}\\
		&&+(-2 k_1 k_4 l_1-2 k_1 k_3 l_2+2 k_1 k_2 l_3+\left(k_1^2+k_2^2\right) l_4)e_{3},
	\end{eqnarray*}
	it follows that condition (\ref{E6}) holds for $\nu, \nu, g\nu$. Similar computations for other cases follow analogous reasoning.
\end{proof}

For the twisted partial action $(\mathbb{H}_{4}, \mathbb{H}_{ss},\cdot, \omega)$, by Lemma \ref{F1}, we have  the partial crossed product $\mathbb{H}_{ss}\sharp_{(\cdot,\omega)}\mathbb{H}_{4}$. Computing its basis elements, we find: 
$$
\begin{cases}
	1\sharp 1=1\otimes 1, e_{1}\sharp 1=e_{1}\otimes 1,e_{2}\sharp 1=e_{2}\otimes 1,e_{3}\sharp 1=e_{3}\otimes 1,&\\
	1\sharp g=e_{1}\sharp g=e_{2}\sharp g=e_{3}\sharp g=0,&\\
	1\sharp \nu=k_{1}1\otimes 1+k_{2}e_{1}\otimes 1+k_{3}e_{2}\otimes 1-k_{4}e_{3}\otimes 1, e_{3}\sharp \nu=k_{1}e_{3}\otimes 1-k_{2}e_{2}\otimes 1,&\\
	e_{1}\sharp \nu=k_{1}e_{1}\otimes 1+k_{2}1\otimes 1+k_{3}e_{3}\otimes 1-k_{4}e_{2}\otimes 1,e_{2}\sharp \nu=k_{1}e_{2}\otimes 1-k_{2}e_{3}\otimes 1,&\\
	1\sharp g\nu=1\otimes g\nu+ l_{1}1 \otimes g+ l_{2}e_{1} \otimes g+ l_{3}e_{2} \otimes g+ l_{4}e_{3} \otimes g,&\\
		e_{1}\sharp g\nu=e_{1}\otimes g\nu+ l_{1}e_{1} \otimes g+ l_{2}1 \otimes g+ l_{3}e_{3} \otimes g+ l_{4}e_{2} \otimes g,&\\
e_{2}\sharp g\nu=e_{2}\otimes g\nu+ l_{1}e_{2} \otimes g- l_{2}e_{3} \otimes g, e_{3}\sharp g\nu=e_{3}\otimes g\nu+ l_{1}e_{3} \otimes g- l_{2}e_{2} \otimes g,
\end{cases}
$$
and it follows that $1\sharp 1, e_{1}\sharp 1,e_{2}\sharp 1,e_{3}\sharp 1,$
$1\sharp g\nu$, $e_{1}\sharp g\nu$,$e_{2}\sharp g\nu$, $e_{3}\sharp g\nu$ is a basis for the vector space $\mathbb{H}_{ss}\sharp_{(\cdot, \omega)} \mathbb{H}_{4}$.
The multiplication of the partial crossed product $\mathbb{H}_{ss}\sharp_{(\cdot, \omega)} \mathbb{H}_{4}$ is  given by 
$$
(a\sharp h)(b\sharp l)=a(h_{1}\cdot b)\omega(h_{2},l_{1})\sharp h_{3}l_{2},
$$
for all $a, b\in \mathbb{H}_{ss}$ and $h, l\in \mathbb{H}_{4}$. For example, 
\begin{eqnarray*}
	(e_{1}\sharp \nu)	(e_{2}\sharp \nu)
	&=&e_{1}(\nu_{1}\cdot e_{2})\omega(\nu_{2},\nu_{1})\sharp \nu_{3}\nu_{2}\\
	&=&e_{1}(g\cdot e_{2})\omega(g,g)\sharp \nu\nu+e_{1}(g\cdot e_{2})\omega(g,\nu)\sharp \nu+e_{1}(g\cdot e_{2})\omega(\nu,g)\sharp \nu\\
	&&+e_{1}(g\cdot e_{2})\omega(\nu,\nu)\sharp 1+e_{1}(\nu\cdot e_{2})\omega(1,g)\sharp \nu+e_{1}(\nu\cdot e_{2})\omega(1,\nu)\sharp 1\\
	&=&e_{1}(\nu\cdot e_{2})\omega(1,\nu)\sharp 1\\
	&=&e_{1}(k_{1}e_{2}+k_{2}e_{3})(k_{1}1+k_{2}e_{1}+k_{3}e_{2}-k_{4}e_{3})\sharp 1\\
	&=&(k_{1}^{2}-k_{2}^{2})e_{3}\sharp 1.
\end{eqnarray*}

\section{Twisted partial actions of $\mathbb{H}_{4}$ on $\mathbb{H}_{s}$}

Extending the framework to the split quaternion algebra $ \mathbb{H}_{s}$, define two linear maps $\cdot: \mathbb{H}_{4}\otimes \mathbb{H}_{s}\rightarrow  \mathbb{H}_{s} $ and $\omega: \mathbb{H}_{4}\otimes\mathbb{H}_{4}\rightarrow \mathbb{H}_{s}$ as follows: 
$$\resizebox{1\hsize}{!}{$\begin{array}{|c|c|c|c|c|}
		\hline   \cdot                           & 1 & e_{1} &e_{2}&e_{3} \\
		\hline 1                 & 1 & e_{1} &e_{2}&e_{3} \\
		\hline  g                     & 0& 0 & 0 & 0 \\
		\hline  \nu                   &e_{3}& e_{2}&  e_{1} & 1 \\
		\hline  g\nu                    & l_{1}1-l_{2}e_{1}+l_{3}e_{2}+l_{4}e_{3}&l_{2}1+l_{1}e_{1}-l_{4}e_{2}+l_{3}e_{3}& l_{3}1-l_{4}e_{1}+l_{1}e_{2}+l_{2}e_{3}& l_{4}1+l_{3}e_{1}-l_{2}e_{2}+l_{1}e_{3}  \\
		\hline
	\end{array}$}$$

$$
\begin{cases}
	\omega(1,1)=1, \omega(1,g)=0, \omega(1,\nu)=e_{3}, \omega(1,g\nu)=l_{1}1-l_{2}e_{1}+l_{3}e_{2}+l_{4}e_{3},&\\
	\omega(g,1)=	\omega(g,g)=	\omega(g,\nu)=	\omega(g,g\nu)=0, &\\
	\omega(\nu,1)=e_{3}, \omega(\nu,g)=0,
	\omega(\nu,\nu)=1,\\ \omega(\nu,g\nu)=l_{4}1+l_{3}e_{1}-l_{2}e_{2}+l_{1}e_{3}, &\\
	\omega(g\nu,1)=l_{1}1-l_{2}e_{1}+l_{3}e_{2}+l_{4}e_{3}, \omega(g\nu,g)=\omega(g\nu,g\nu)=0,\\ \omega(g\nu,\nu)=l_{4}1+l_{3}e_{1}-l_{2}e_{2}+l_{1}e_{3}, &
\end{cases}
$$
where $ l_{j}( j=1,2,3,4)$ are parameters. 

\begin{theorem}
	With  $\cdot$ and $\omega$ defined above, the pair $(\cdot, \omega)$ is also a twisted partial action of $\mathbb{H}_{4}$ on $\mathbb{H}_{s}$. 
\end{theorem}
\begin{proof}
	 The verification follows similar steps as in Section 3.
\end{proof}

\section{Twisted partial actions of $\mathbb{H}_{4}$ on $\mathbb{H}_{00}$}

Define two linear maps $\cdot: \mathbb{H}_{4}\otimes \mathbb{H}_{00}\rightarrow  \mathbb{H}_{00} $ and $\omega: \mathbb{H}_{4}\otimes\mathbb{H}_{4}\rightarrow \mathbb{H}_{00}$ as follows: 
$$\begin{array}{|c|c|c|c|c|}
	\hline   \cdot                           & 1 & e_{1} &e_{2}&e_{3} \\
	\hline 1                 & 1 & e_{1} &e_{2}&e_{3} \\
	\hline  g                     & 0& 0 & 0 & 0 \\
	\hline  \nu                   &k_{1}e_{1}+k_{2}e_{2}+k_{3}e_{3}& -k_{2}e_{3}&  k_{1}e_{3} & 0  \\
	\hline  g\nu                    & l_{1}e_{0}+l_{2}e_{1}+l_{3}e_{2}+l_{4}e_{3}&l_{1}e_{1}+l_{3}e_{3}& l_{1}e_{2}-l_{2}e_{3}& l_{1}e_{3}    \\
	\hline
\end{array}$$
$$
\begin{cases}
	\omega(1,1)=1, \omega(1,g)=0, \omega(1,\nu)=k_{1}e_{1}+k_{2}e_{2}+k_{3}e_{3},\\ \omega(1,g\nu)=l_{1}1+l_{2}e_{1}+l_{3}e_{2}+l_{4}e_{3},&\\
	\omega(g,1)=	\omega(g,g)=	\omega(g,\nu)=	\omega(g,g\nu)=0, &\\
	\omega(\nu,1)=k_{1}e_{1}+k_{2}e_{2}+k_{3}e_{3}, \omega(\nu,g)=\omega(\nu,\nu)=0,\\ \omega(\nu,g\nu)=k_{1}l_{1}e_{1}+k_{2}l_{1}e_{2}+(k_{3}l_{1}-k_{2}l_{2}+k_{1}l_{3})e_{3}, &\\
	\omega(g\nu,1)=l_{1}1+l_{2}e_{1}+l_{3}e_{2}+l_{4}e_{3}, \omega(g\nu,g)=\omega(g\nu,g\nu)=0,\\ \omega(g\nu,\nu)=k_{1}l_{1}e_{1}+k_{2}l_{1}e_{2}+(k_{3}l_{1}-k_{2}l_{2}+k_{1}l_{3})e_{3}, &
\end{cases}
$$
where $k_{i}, l_{j}(i=1,2,3, j=1,2,3,4)$ are parameters. 
\begin{theorem}
	With  $\cdot$ and $\omega$ defined above, the pair $(\cdot, \omega)$ is also a twisted partial action of $\mathbb{H}_{4}$ on $\mathbb{H}_{00}$. 
\end{theorem}
\begin{proof}
	The verification follows similar steps as in Section 3.
\end{proof}

\section{Conclusion}
This paper establishes a class of twisted partial Hopf actions on split quaternions, split semi-quaternions, and $\frac{1}{4}$-quaternions. By leveraging the structure of the Sweedler Hopf algebra $ H_4 $, we provide a systematic framework for studying these non-commutative algebras. Future work will explore applications in quantum symmetry and non-commutative geometry.
\def\theequation{5. \arabic{equation}}
\setcounter{equation} {0} \hskip\parindent

\end{document}